\documentclass{amsart}
\usepackage[all]{xy}
\usepackage{verbatim}
\usepackage{color}
\usepackage{amsthm}
\usepackage{amssymb}
\usepackage[colorlinks=true]{hyperref}

\numberwithin{equation}{section}

\newtheorem{theorem}[equation]{Theorem}
\newtheorem*{theorem*}{Theorem} \newtheorem{lemma}[equation]{Lemma}

\newtheorem*{conjecture*}{Mamma Conjecture}
\newtheorem*{conjecture1*}{Mamma Conjecture (revisited)}
\newtheorem{proposition}[equation]{Proposition}
\newtheorem{corollary}[equation]{Corollary}
\newtheorem*{corollary*}{Corollary}

\theoremstyle{remark}
\newtheorem{definition}[equation]{Definition}

\newtheorem{notation}[equation]{Notation}

\theoremstyle{remark}
\newtheorem{remark}[equation]{Remark}

\newcommand{\cA}{{\mathcal A}}
\newcommand{\cB}{{\mathcal B}}
\newcommand{\cC}{{\mathcal C}}
\newcommand{\cD}{{\mathcal D}}
\newcommand{\cE}{{\mathcal E}}
\newcommand{\cF}{{\mathcal F}}

\newcommand{\cO}{{\mathcal O}}

\newcommand{\Spt}{\mathrm{Spt}}

\newcommand{\bbA}{\mathbb{A}}

\newcommand{\bbC}{\mathbb{C}}

\newcommand{\bbH}{\mathbb{H}}
\newcommand{\bbL}{\mathbb{L}}

\newcommand{\bbP}{\mathbb{P}}
\newcommand{\bbR}{\mathbb{R}}

\newcommand{\bbQ}{\mathbb{Q}}
\newcommand{\bbZ}{\mathbb{Z}}

\DeclareMathOperator{\id}{id}

\DeclareMathOperator{\Mod}{Mod}

\DeclareMathOperator{\Fun}{Fun} 

\newcommand{\dgcat}{\mathsf{dgcat}} 
\newcommand{\Az}{\mathsf{Az}} 

\newcommand{\bbK}{I\mspace{-6.mu}K}

\newcommand{\perf}{\mathrm{perf}}

\newcommand{\dg}{\mathsf{dg}}

\newcommand{\Hom}{\mathrm{Hom}}
\newcommand{\End}{\mathrm{End}}

\newcommand{\rep}{\mathrm{rep}}

\newcommand{\dgHo}{\mathrm{H}^0}

\newcommand{\Ho}{\mathrm{Ho}}

\newcommand{\Hmo}{\mathsf{Hmo}}
\newcommand{\op}{\mathrm{op}}

\newcommand{\too}{\longrightarrow}

\newcommand{\ie}{\textsl{i.e.}\ }
\newcommand{\eg}{\textsl{e.g.}}

\begin{document}

\title[Noncommutative motives of Azumaya algebras]{Noncommutative motives of Azumaya algebras}
\author{Gon{\c c}alo~Tabuada and Michel Van den Bergh}

\address{Gon{\c c}alo Tabuada, Department of Mathematics, MIT, Cambridge, MA 02139, USA}
\email{tabuada@math.mit.edu}
\urladdr{http://math.mit.edu/~tabuada}
\thanks{G.~Tabuada was partially supported by the NEC Award-2742738}

\address{Michel Van den Bergh, Departement WNI, Universiteit Hasselt, 3590 Diepenbeek, Belgium}
\email{michel.vandenbergh@uhasselt.be} 
\urladdr{http://hardy.uhasselt.be/personal/vdbergh/Members/~michelid.html}
\thanks{M.~Van den Bergh is a Director of Research at the FWO Flanders}

\thanks{This material is based upon work supported by the National Science 
Foundation (NSF) under the Grant No.\ 0932078 000, while the authors were in 
residence at the Mathematical Science Research Institute (MSRI) in 
Berkeley, California, during the Spring semester of 2013.}

\date{\today}


\abstract{Let $k$ be a base commutative ring, $R$ a commutative ring of coefficients, $X$ a quasi-compact quasi-separated $k$-scheme with $m$ connected components, $A$ a sheaf of Azumaya algebras over $X$ of rank $(r_1, \ldots, r_m)$, and $\Hmo_0(k)_R$ the category of noncommutative motives with $R$-coefficients. Assume that $1/r \in R$ with $r:=r_1 \times \cdots \times r_m$. Under these assumptions, we prove that the noncommutative motives with $R$-coefficients of $X$ and $A$ are isomorphic. As an application, we show that all the $R$-linear additive invariants of $X$ and $A$ are exactly the same. Examples include (nonconnective) algebraic $K$-theory, cyclic homology (and all its variants), topological Hochschild homology, etc. Making use of these isomorphisms, we then compute the $R$-linear additive invariants of differential operators in positive characteristic, of cubic fourfolds containing a plane, of Severi-Brauer varieties, of Clifford algebras, of quadrics, and of finite dimensional $k$-algebras of finite global dimension. Along the way we establish two results of independent interest. The first one asserts that every element $\alpha \in K_0(X)$ of rank $(r_1, \ldots, r_m)$ becomes invertible in the $R$-linearized Grothendieck group $K_0(X)_R$, and the second one that every additive invariant of finite dimensional algebras of finite global dimension is unaffected under nilpotent extensions.
}}

\maketitle
\vskip-\baselineskip
\vskip-\baselineskip
\vskip-\baselineskip



\section{Introduction}
\subsection*{Azumaya algebras}
Sheaves of Azumaya algebras over schemes $X$ were introduced in the late sixties by Grothendieck \cite{Grothendieck}. Formally, a sheaf $A$ of $\cO_X$-algebras is {\em Azumaya} if it is locally free of finite rank over $\cO_X$  and the canonical morphism
$$A^\op \otimes_{\cO_X}A \stackrel{\sim}{\too} \Hom_{\cO_X}(A,A)$$ is an isomorphism.
Locally, for the {\'e}tale topology, $A$ is simply a matrix algebra. This generalizes the notion of an Azumaya algebra over a commutative 
ring \cite{Auslander-Goldman, Azumaya} and consequently the notion of a central simple algebra over a field.
\subsection*{Noncommutative motives}
A {\em differential graded (=dg) category} $\cA$, over a base
commutative ring $k$, is a category enriched over complexes of
$k$-modules; see \S\ref{sec:dg}. Every (dg) $k$-algebra $S$ gives rise
to a dg category $\underline{S}$ with a single object and (dg)
$k$-algebra of endomorphisms $S$. In the same vein, every
quasi-compact quasi-separated $k$-scheme $X$ gives rise to a
(canonical) dg category $\perf_\dg(X)$ which enhances the derived
category $\perf(X)$ of perfect complexes of $\cO_X$-modules; consult
\S\ref{sec:schemes} for details. Let us denote by $\dgcat(k)$ the
category of (small) dg categories.

Classical invariants such as algebraic $K$-theory ($K$), nonconnective
algebraic $K$-theory ($\bbK$), Hochschild homology ($HH$), cyclic
homology ($HC$), periodic cyclic homology ($HP$), negative cyclic
homology ($HN$), topological Hochschild homology ($THH$), and
topological cyclic homology ($TC$), extend naturally from $k$-algebras
to dg categories. In order to study all these invariants
simultaneously the notion of additive invariant was introduced in
\cite{Additive}. Let us now recall it. Given a dg category $\cA$, let $T(\cA)$ be the dg category of pairs $(i,x)$, where $ i \in \{1,2\}$ and $x$ is an object of $\cA$. The complex of morphisms in $T(\cA)$ from $(i,x)$ to $(i',x')$ is given by $\cA(x,x')$ if $ i' \geq i$ and is zero otherwise. Composition is induced by $\cA$; consult \cite[\S4]{Additive} for details. Intuitively speaking, $T(\cA)$ ``dg categorifies'' the notion of upper triangular matrix. Note that we have two inclusion dg functors $i_1:\cA \hookrightarrow T(\cA)$ and $i_2:\cA\hookrightarrow T(\cA)$. A functor $E:\dgcat(k) \to \mathsf{D}$, with values in an additive category, is called an {\em additive invariant} if it satisfies the following two conditions:
\begin{itemize}
\item[(i)] it sends {\em Morita equivalences} (see \S\ref{sec:dg}) to isomorphisms;
\item[(ii)] given a dg category $\cA$, the inclusion dg functors induce an isomorphism\footnote{Condition (ii) can be equivalently formulated in terms of semi-orthogonal decompositions in the
sense of Bondal-Orlov~\cite{BO}; see \cite[Thm.~6.3(4)]{Additive}.}
$$ [E(i_1)\,\, E(i_2)]: E(\cA) \oplus E(\cA) \stackrel{\sim}{\too} E(T(\cA))\,.$$
\end{itemize}
Thanks to the work of
Blumberg-Mandell, Keller, Quillen, Schlichting, Thomason-Trobaugh, Waldhausen,
and others (see \cite{BM,Exact,Exact2,Quillen, Negative, Spectral,TT,Wald}),
all the above invariants are additive. Moreover, when applied to
$\underline{S}$, resp. to $\perf_\dg(X)$, they agree with the
classical invariants of (dg) $k$-algebras, resp. of $k$-schemes. 

Let $R$ be a commutative ring of coefficients. In \cite{Additive} the {\em universal 
additive invariant} (with $R$-coefficients) was constructed
\begin{equation}\label{eq:universal}
U(-)_R:\dgcat(k) \too \Hmo_0(k)_R\,.
\end{equation}
Given any $R$-linear additive category $\mathsf{D}$, there is an induced equivalence of categories
\begin{equation}\label{eq:categories}
U(-)_R^\ast: \Fun_{R\text{-}\mathrm{linear}}(\Hmo_0(k)_R,\mathsf{D}) \stackrel{\sim}{\too} \Fun_{\mathrm{Additive}}(\dgcat(k),\mathsf{D})\,,
\end{equation}
where the left-hand-side denotes the category of additive $R$-linear functors and the right-hand-side the category of additive invariants. Because of this universal property, which is reminiscent from motives, $\Hmo_0(k)_R$ is called the category of {\em noncommutative motives} (with $R$-coefficients); consult \S\ref{sec:NCmotives} for further details. The tensor product of $k$-algebras extends also naturally to dg categories, giving thus rise to a symmetric monoidal structure $-\otimes-$ on $\dgcat(k)$. After deriving it, this structure descends to $\Hmo_0(k)_R$ and makes the functor \eqref{eq:universal} symmetric monoidal.
\subsection*{Motivation}
Let $X$ be a quasi-compact quasi-separated $k$-scheme and $A$ a sheaf of Azumaya algebras over $X$. Similarly to $\perf_\dg(X)$, one can construct the dg category $\perf_\dg(A)$ of perfect complexes of $A$-modules; see \S\ref{sec:schemes}. This dg category reduces to $\perf_\dg(X)$ when $A=\cO_X$ and comes equipped with a canonical dg functor $-\otimes_{\cO_X}A: \perf_\dg(X) \to \perf_\dg(A)$. One obtains in this way two well-defined noncommutative motives 
\begin{eqnarray}\label{eq:NCmotives}
U(\perf_\dg(X))_R &  & U(\perf_\dg(A))_R\,.
\end{eqnarray}
As mentioned above, $A$ is {\'e}tale-locally a matrix algebra. Hence, up to an {\'e}tale covering of $X$, $A$ and $\cO_X$ are Morita equivalent. This leads naturally to the following motivating question: {\it How ``close'' are the above noncommutative motives \eqref{eq:NCmotives} ?} 

In this article we provide a precise and complete answer to this question. As a by-product we obtain several applications of general interest; see \S\ref{sec:applications}.
\section{Statement of results}
Let $k$ be a base commutative ring and $R$ a commutative ring of coefficients. Recall that a scheme $X$ is {\em
  quasi-compact} if it admits a finite covering by affine open
subschemes, and {\em quasi-separated} if the intersection of any two
affine open subschemes is quasi-compact. Note that every such
scheme $X$ has always a finite number of connected components. Our
main result, which answers the above motivating question, is the
following:
\begin{theorem}\label{thm:main}
Let $X$ be a quasi-compact quasi-separated $k$-scheme with $m$ connected components, and $A$ a sheaf of Azumaya algebras over $X$ of rank $(r_1, \ldots, r_m)$. Assume that $1/r \in R$ with $r:=r_1 \times \cdots \times r_m$. Under these assumptions, one has the following isomorphism 
\begin{eqnarray}\label{eq:isom-main}
U(-\otimes_{\cO_X}A)_R:U(\perf_\dg(X))_R \stackrel{\sim}{\too} U(\perf_\dg(A))_R\,.
\end{eqnarray}
\end{theorem}
Theorem~\ref{thm:main} shows us that the difference between the noncommutative motives \eqref{eq:NCmotives} is simply a $r$-torsion phenomenon. As explained in Remark~\ref{rk:optimal} below, this result is {\em optimal}, \ie it does not hold without the assumption $1/r \in R$.

In order to prove Theorem~\ref{thm:main} we have established a
$K$-theoretical result, which is of independent interest. Recall from \cite[page~71]{Weibel} that given a scheme $X$ with $m$ connected
components, one has a well-defined (split surjective) ring
homomorphism $\mathrm{rank}:K_0(X) \twoheadrightarrow \bbZ^m$.  Let
us write $I_X$ for its kernel. Whenever $X$ is Noetherian\footnote{Note that
  every noetherian scheme is quasi-compact and quasi-separated.}, of Krull
dimension $d$, and admits an ample sheaf, we have $I_X^{d+1}=0$; see
\cite[\S V\ Cor.\ 3.10]{Lang}. If one does not require a uniform bound on the order of nilpotency of the elements in~$I_X$ then this result may be generalized as follows:
 \begin{theorem}\label{thm:main2}
Let $X$ be a quasi-compact quasi-separated scheme $X$. Under this assumption, every element in $K_0(X)$ of rank zero is nilpotent.
\end{theorem}
This statement appears not to exist in the literature. The affine case was proved by Gabber in \cite[page 188]{Gabber} using absolute noetherian approximation.\footnote{Ben Antieau \cite{Antieau} has indicated to us how Gabber's approach can be extended to the general case by using absolute noetherian approximation for quasi-compact quasi-separated schemes
and the local global spectral sequence for nonconnective K-theory. Our argument uses simply Mayer-Vietoris and does not depend on absolute noetherian approximation.}

Making use of Theorem~\ref{thm:main2}, one obtains the following useful invertibility result:  
\begin{corollary}\label{cor:main2} 
Let $X$ be as in Theorem~\ref{thm:main2} and $\alpha$ an element in $K_0(X)$ of rank $(r_1, \ldots, r_m)$. Assume that $1/r\in R$ with $r:=r_1 \times \cdots \times r_m$. Under these assumptions, the image of $\alpha$  in $K_0(X)_R$ is invertible.
\end{corollary}
\section{Applications}\label{sec:applications}
\subsection*{Additive invariants}
Let $k$ be a base commutative ring. As
explained above, all the classical invariants of quasi-compact quasi-separated $k$-schemes $X$ can be recovered from  the dg category $\perf_\dg(X)$. Hence, given a sheaf $A$ of Azumaya algebras over $X$ and an additive invariant $E:\dgcat(k)\to \mathsf{D}$, let
us write $E(A)$ for the value of $E$ at
$\perf_\dg(A)$. By combining Theorem
\ref{thm:main} with the above equivalence \eqref{eq:categories} of
categories one obtains the following result:
\begin{corollary}\label{cor:main}
  Let $X, A, r, R$ be as in Theorem~\ref{thm:main}, and $E:\dgcat(k)
  \to \mathsf{D}$ an additive invariant with values in a $R$-linear category. Under these assumptions, one has an isomorphism $E(X)\simeq E(A)$.
\end{corollary}
When applied to the above examples of additive invariants, Corollary~\ref{cor:main} gives rise to the following (concrete) isomorphisms
\begin{eqnarray}
K_\ast(X)_{1/r} \simeq K_\ast(A)_{1/r} && \bbK_\ast(X)_{1/r} \simeq \bbK_\ast(A)_{1/r} \label{eq:isom1} \\
THH_\ast(X)_{1/r} \simeq THH_\ast(A)_{1/r} && TC_\ast(X)_{1/r} \simeq TC_\ast(A)_{1/r} \label{eq:isom2}\,,
\end{eqnarray}
where $(-)_{1/r}:=(-)_{\bbZ[1/r]}$. When $1/r\in k$ one has moreover the isomorphisms
\begin{eqnarray}
HH_\ast(X) \simeq HH_\ast(A) && HC_\ast(X) \simeq HC_\ast(A) \label{eq:isom3} \\
HP_\ast(X) \simeq HP_\ast(A) && HN_\ast(X) \simeq HN_\ast(A) \label{eq:isom4}\,.
\end{eqnarray}
\begin{remark}
  By considering $K, \bbK, THH$ and $TC$ as spectra-valued functors
  one observes that the isomorphisms \eqref{eq:isom1}-\eqref{eq:isom2}
  can be lifted to the homotopy category of spectra localized at the
  $\bbZ[1/r]$-linear stable equivalences. In the same vein, the
  isomorphisms \eqref{eq:isom3}-\eqref{eq:isom4} admit a lifting to
  the derived category of mixed complexes; consult Keller's survey
  \cite[\S5.3]{ICM-Keller} for details
\end{remark}
The isomorphism $HH_\ast(X)\simeq HH_\ast(A)$ is well-known and holds
without the assumption $1/r \in k$. In what concerns cyclic homology,
the isomorphism $HC_\ast(X)\simeq HC_\ast(A)$ was established by
Corti{\~n}as-Weibel \cite{CW} in the affine case.  The algebraic
$K$-theory isomorphism $K_\ast(X)_{1/r}\simeq K_\ast(A)_{1/r}$ was
obtained recently by Hazrat-Hoobler \cite{HH} under the assumption
that $X$ is either regular noetherian or noetherian of finite Krull
dimension with an ample sheaf. Besides these particular cases, all the
remaining isomorphisms provided by Corollary~\ref{cor:main} are, to
the best of the authors knowledge, new in the literature.
\begin{remark}\label{rk:optimal}
  The above isomorphism \eqref{eq:isom-main} is {\em optimal}, \ie it
  does not hold without the assumption $1/r \in R$. An example is
  given by the quartenions $\bbH$, considered as a central simple
  $\bbR$-algebra of dimension $4$. Recall from \cite[page~181]{Weibel}
  that $K_1(\bbR)\simeq \bbR^\times$ and $K_1(\bbH)\simeq
  \bbR^\times_+$.  These groups are not (abstractly) isomorphic since
  $\bbR^\times\simeq \bbR^\times_+ \times \bbZ/2\bbZ$ has $2$-torsion
  while $\bbR^\times_+$ is torsion-free. Using
  Corollary~\ref{cor:main} one then concludes that
  $U(\underline{\bbR})_R \not\simeq U(\underline{\bbH})_R$ whenever
  $1/2 \notin R$. Instead of $K_1$ one may also use other higher
  $K$-groups to obtain counterexamples. Indeed, as explained in
  \cite[page~475]{Weibel}, $K_2(\bbR)$ has $2$-torsion and $K_2(\bbH)$
  is torsion-free, and $K_5(\bbR),K_6(\bbR)$ are torsion-free and $K_5(\bbH),K_6(\bbH)$ have $2$-torsion.  
\end{remark}
\subsection*{Differential operators in positive characteristic}
Let $k$ be an algebraically closed field of characteristic $p >0$, $X$
a smooth $k$-scheme\footnote{In particular, $X$ is quasi-compact and
  separated.}, $T^\ast X^{(1)}$ the Frobenius twist of the total
cotangent bundle of $X$, and $\cD_X$ the sheaf of (crystalline)
differential operators on $X$; consult \cite[\S1]{Roman} for
details. As proved by Bezrukavnikov-Mirkovi{\'c}-Rumynin in
\cite[Thm.~2.2.3]{Roman}, $\cD_X$ is a sheaf of Azumaya algebras over
$T^\ast X^{(1)}$ of rank $p^{2 \mathrm{dim}(X)}$. In the particular
case where $X$ is affine space $\bbA^n:= \mathrm{Spec}(k[x_1,
\ldots, x_n])$, $\cD_X$ reduces to the Weyl algebra ($\partial_i:=\partial/\partial x_i$)
\begin{eqnarray*}
k\langle x_1, \ldots, x_n,\partial _1, \ldots, \partial _n\rangle && [\partial _i, x_j]=\delta_{ij}
\end{eqnarray*}
and $T^\ast X^{(1)}$ to polynomials in $2n$ variables $k[x_1^p,\ldots, x_n^p,\partial _1^p, \ldots, \partial _n^p]$; consult \cite[page~951]{Roman} as well as Revoy's work \cite{Revoy}. Thanks to Theorem~\ref{thm:main} (with $X=T^\ast X^{(1)}$ and $A=\cD_X$), one hence obtains a motivic isomorphism
$$ U(\perf_\dg(T^\ast X^{(1)}))_R \simeq U(\perf_\dg(\cD_X))_R$$
for every commutative ring $R$ containing $1/p$.
\begin{corollary}
Let $k,X, R$ be as above, and $E: \dgcat(k) \to \mathsf{D}$ an additive invariant with values in a $R$-linear category. Under these assumptions, one has an isomorphism $E(T^\ast X^{(1)})\simeq E(\cD_X)$.
\end{corollary}
\subsection*{Cubic fourfolds containing a plane}
Let $k=\bbC$ and $X$ a (generic) cubic fourfold, \ie a smooth complex hypersurface of degree $3$ in $\bbP^5$. In the case where $X$ contains a plane, Kuznetsov constructed in \cite{Kuznetsov-cubic} a semi-orthogonal decomposition
$$ \perf(X) = (\perf(B_S),\cO_X,\cO_X(1),\cO_X(2))\,,$$
where $S$ is a smooth projective complex $K3$-surface and $B_S$ a sheaf of Azumaya algebras over $S$ of rank $4$. By combining Theorem~\ref{thm:main} (with $X=S$ and $A=B_S$) with \cite[Lem.~5.1]{Exceptional}, one hence obtains the following motivic decomposition
$$ U(\perf_\dg(X))_R \simeq U(\perf_\dg(S))_R \oplus U(\underline{\bbC})_R^{\oplus 3}$$
for every commutative ring $R$ containing $1/2$.
\begin{corollary}
Let $X,S, R$ be as above, and $E:\dgcat(\bbC) \to \mathsf{D}$ an additive invariant with values in a $R$-linear category. Under these assumptions, one has an isomorphism $E(X)\simeq E(S)\oplus E(\bbC)^{\oplus 3}$.
\end{corollary}
\subsection*{Severi-Brauer varieties}
Let $k$ be a field, $A$ a central simple $k$-algebra of degree $\sqrt{\mathrm{dim}(A)}=d$, and $SB(A)$ the associated Severi-Brauer variety. As proved in \cite[Prop.~2.8]{MT}, one has the following motivic decomposition
\begin{equation}\label{eq:decomposition}
U(\perf_\dg(SB(A)))_R \simeq U(\underline{k})_R \oplus U(\underline{A})_R \oplus U(\underline{A})^{\otimes 2}_R \oplus \cdots \oplus U(\underline{A})_R^{\otimes d-1}
\end{equation}
for every commutative ring $R$. Consequently, Theorem~\ref{thm:main} (with $X=\mathrm{Spec}(k)$) combined with the fact that $U(\underline{k})_R$ is the $\otimes$-unit of $\Hmo_0(k)_R$, allows us to conclude that whenever $1/d \in R$, \eqref{eq:decomposition} reduces to
$$ U(\perf_\dg(SB(A)))_R \simeq \underbrace{U(\underline{k})_R \oplus \cdots \oplus U(\underline{k})_R}_{d\text{-}\mathrm{copies}}\,.$$
\begin{corollary}\label{cor:SB}
Let $A, R$ be as above, and $E:\dgcat(k) \to \mathsf{D}$ an additive invariant with values in a $R$-linear category. Under these assumptions, one has an isomorphism $E(SB(A))\simeq E(k)^{\oplus d}$. 
\end{corollary}
\subsection*{Clifford algebras}
Let $k$ be a field of characteristic $\neq 2$, $V$ a finite dimensional $k$-vector space of dimension $n$, and $q:V\to k$ a non-degenerate quadratic form. Recall from \cite[\S V]{Lam} that out of this data one can construct the Clifford algebra $C(q)$, the even Clifford algebra $C_0(q)$, and the signed determinant $\delta(q)\in k^\times/(k^\times)^2$. The $k$-algebra $C(q)$ has dimension $2^n$ and $C_0(q)$ dimension $2^{n-1}$. When $n$ is odd we have the following structure results (see \cite[\S V Thm.~2.4]{Lam}): 
\begin{itemize}
\item[(i)] $C_0(q)$ is a central simple $k$-algebra; 
\item[(ii)] When $\delta(q) \notin (k^\times)^2$, $C(q)$ is a central simple algebra over its center $k(\sqrt{\delta(q)})$; 
\item[(iii)] When $\delta(q) \in (k^\times)^2$, $C(q)$ is a product of two isomorphic central simple algebras over the center $k \times k$. 
\end{itemize}
Using Theorem \ref{thm:main} we then obtain the following motivic decompositions
\begin{eqnarray*}
U(\underline{C_0(q)})_R\simeq U(\underline{k})_R && U(\underline{C(q)})_R \simeq \left\{
  \begin{array}{lcr}
    U(\underline{k(\sqrt{\delta(q)})})_R & \mathrm{when} & \delta(q) \notin (k^\times)^2  \\
    U(\underline{k})_R\oplus  U(\underline{k})_R & \mathrm{when} & \delta(q) \in (k^\times)^2 \\
  \end{array}
\right.
\end{eqnarray*}
for every commutative ring $R$ containing $1/2$. When $n$ is even we have the (opposite) structure results (see \cite[\S V Thm.~2.5]{Lam}): 
\begin{itemize}
\item[(i')] $C(q)$ is a central simple $k$-algebra; 
\item[(ii')] When $\delta(q) \notin (k^\times)^2$, $C_0(q)$ is a central simple algebra over its center $k(\sqrt{\delta(q)})$; 
\item[(iii')] When $\delta(q) \in (k^\times)^2$, $C_0(q)$ is a product of two isomorphic central simple algebras over the center $k \times k$.
\end{itemize}
Using Theorem \ref{thm:main} once again we obtain the motivic decompositions
\begin{eqnarray*}
 U(\underline{C(q)})_R\simeq U(\underline{k})_R && U(\underline{C_0(q)})_R \simeq \left\{
  \begin{array}{lcr}
    U(\underline{k(\sqrt{\delta(q)})})_R & \mathrm{when} & \delta(q) \notin (k^\times)^2  \\
    U(\underline{k})_R\oplus  U(\underline{k})_R & \mathrm{when} & \delta(q) \in (k^\times)^2 \\
  \end{array}
\right.
\end{eqnarray*}
for every commutative ring $R$ containing $1/2$. Thanks to Corollary~\ref{cor:main}, the above four isomorphisms hold also with $U$ replaced by any additive invariant $E$ with values in a $R$-linear category. 
\subsection*{Quadrics}
Let $k, q$ be as in the previous subsection (with $n \geq 3$), and $Q_q \subset \bbP(V)$ the associated smooth projective quadric of dimension $n-2$. As explained in the proof of \cite[Prop.~2.3]{MT}, one has the following motivic decomposition
\begin{equation*}
U(\perf_\dg(Q_q))_R \simeq U(\underline{C_0(q)})_R \oplus \underbrace{U(\underline{k})_R \oplus \cdots \oplus U(\underline{k})_R}_{(n-2)\text{-}\mathrm{copies}} 
\end{equation*}
for every commutative ring $R$. By combing it with Corollary~\ref{cor:main} and with the four isomorphisms of the previous subsection, we obtain the following result:
\begin{corollary}\label{cor:D-modules}
  Let $k, q$ be as above, $R$ a commutative ring containing
  $1/2$, and $E:\dgcat(k)\to \mathsf{D}$ an additive invariant with
  values in a $R$-linear category. Under these assumptions, one has an isomorphism between $E(Q_q)$ and:
\begin{itemize}  
\item[(i)] $E(k)^{\oplus n-1}$ when $n$ is odd;
\item[(ii)] $E(k)^{\oplus n}$ when $n$ is even and $\delta(q)
  \in (k^\times)^2$; 
  \item[(iii)] $E(C_0(q))\oplus E(k)^{\oplus n-1}$ when $n$ is even and $\delta(q) \notin (k^\times)^2$.
  \end{itemize}
  \end{corollary}
\subsection*{Finite dimensional algebras of finite global dimension}
Let $k$ be a field of characteristic $p\geq 0$ and $R$ a commutative ring. We start by describing the behavior of the universal additive invariant with respect to nilpotent extensions.
\begin{theorem}{(Nilinvariance)}\label{thm:nilpotence} Let
   $S$ be a finite dimensional $k$-algebra of finite global dimension and $I\subset S$
  a nilpotent (two-sided) ideal. Assume that:
\begin{itemize}
\item[(i)] the $k$-algebra $S/I$ has finite global dimension;
\item[(ii)] the quotient of $S$ by its Jacobson radical $J(S)$ is $k$-separable (e.g. $k$ perfect) or that $1/p\in R$.
\end{itemize}  
Under the above assumptions, one has an induced isomorphism $U(\underline{S})_R
  \stackrel{\sim}{\to} U(\underline{S/I})_R$.
\end{theorem}
In the particular case where $I=J(S)$ the above assumption (i) holds automatically since $S/J(S)$ is semi-simple. Hence, modulo assumption (ii), Theorem~\ref{thm:nilpotence} shows us that the
noncommutative motives of a finite dimensional algebra of finite
global dimension and of its largest semi-simple quotient are
isomorphic.
\begin{corollary}\label{cor:nilpotence}
Let $k, S, I, R$ be as in Theorem~\ref{thm:nilpotence}, and $E:\dgcat(k) \to \mathsf{D}$ an additive invariant with values in an $R$-linear category.
Under these assumptions, one has an isomorphism 
$E(S) \simeq E(S/I)$. In particular, $E(S) \simeq E(S/J(S))$.
\end{corollary}
The isomorphisms $K_n(S) \simeq K_n(S/I), n \geq 0$, are well-known. The case $n=0$ follows from idempotent lifting (see \cite[page~70]{Weibel}) and the remaining cases from d{\'e}vissage (see  \cite[page~118]{Weibel}). All the remaining isomorphisms provided by Corollary~\ref{cor:nilpotence} are, to the best of the authors knowledge, new in the literature.
\begin{remark}
  Theorem~\ref{thm:nilpotence} is {\em false} when $S$ is of infinite
  global dimension. An example is given by the $k$-algebra
  $S:=k[\epsilon]/\epsilon^2$ of dual numbers and by the ideal
  $I:=\epsilon S$. Since $S$ and $S/\epsilon S\simeq k$ are local
  $k$-algebras one has the following isomorphisms
\begin{equation}
\label{eq:K1}
K_1(k[\epsilon]/\epsilon^2) \simeq (k[\epsilon]/\epsilon^2)^\ast = k^\ast + k\epsilon \qquad K_1(S/\epsilon S) \simeq  K_1(k)\simeq k^\ast \,;
\end{equation} 
see \cite[page~183]{Weibel}.  This implies that the induced map
$K_1(S) \stackrel{\epsilon=0}{\to} K_1(S/\epsilon S)$ is not an
isomorphism and so using Corollary~\ref{cor:nilpotence} one concludes
that $U(\underline{S}) \stackrel{\epsilon=0}{\to} U(S/\epsilon S)$ is
not an isomorphism. Note that in the particular case where $k$ is a finite field, the groups \eqref{eq:K1}  are not even abstractly isomorphic because they have
different cardinality. In this case we hence have $U(\underline{S}) \not\simeq
U(S/\epsilon S)$. 
\end{remark}
  Let $V_1,\ldots, V_m$
be the simple (right) $S$-modules and $D_1:= \End_S(V_1), \ldots,
D_m:=\End_S(V_m)$ the associated division $k$-algebras. Thanks to the
Artin-Wedderburn theorem, the quotient $S/J(S)$ is Morita equivalent to
$D_1 \times \cdots \times D_m$. The center $Z_i$ of $D_i$ is a finite
field extension of $k$ and $D_i$ is a central simple
$Z_i$-algebra. Let $r_i:= [D_i:Z_i]$ and $r:=r_1 \times \cdots \times
r_m$. 
 Using Theorem~\ref{thm:main} (with $X=\mathrm{Spec}(Z_i)$ and $A=D_i$), one then obtains the following isomorphism
\begin{equation}\label{eq:iso-sums}
U(\underline{S/J(S)})_R \simeq U(\underline{Z_i})_R \oplus \cdots \oplus U(\underline{Z_m})_R
\end{equation}
for every commutative ring $R$ containing $1/r$ or $1/(rp)$ depending on whether we assume that $S/J(S)$ is $k$-separable or not. The combination of \eqref{eq:iso-sums} with the above Corollary~\ref{cor:nilpotence} gives rise to the following result:
\begin{corollary}\label{cor:fields}
Let $k, S, Z_i, R$ be as above, and $E:\dgcat(k) \to \mathsf{D}$ an additive invariant with values in a $R$-linear category. Under the assumptions of Theorem~\ref{thm:nilpotence}, one has an isomorphism $E(S) \simeq E(Z_1) \oplus \cdots \oplus E(Z_m)$. 
\end{corollary} 
Intuitively speaking, Corollary~\ref{cor:fields} shows us that all additive invariants of finite dimensional $k$-algebras of finite global dimension can be computed using {\em solely} finite field extensions of $k$.
%
\begin{remark}
When $k$ is algebraically closed, we have $D_1 = \cdots = D_m=k$ and consequently $Z_1= \cdots = Z_m=k$. Corollary \ref{cor:fields} reduces then (for every commutative ring $R$) to an isomorphism  $E(S) \simeq E(k)^{\oplus m}$. This isomorphism was also obtained by Keller in \cite[\S2.3]{ilc} using different arguments.
\end{remark}
\subsection*{Notations}\label{sec:notations}
Throughout the article we will reserve the letter $k$ for a base
commutative ring, the letter $R$ for a commutative ring of
coefficients, the letters $A,
B, C$ for sheaves of Azumaya algebras over schemes, the letters
$\cA, \cB$ for dg categories, the letters $S, T$
for (dg) $k$-algebras, and finally the letters $X, Y$ for $k$-schemes. All schemes will be assumed to be
quasi-compact and quasi-separated. Given a small category $\cC$, we will write $\mathrm{Iso}\, \cC$ for its set of isomorphism classes of objects.

\medbreak\noindent\textbf{Acknowledgments:} 
The authors are grateful to the Mathematical Science Research Institute (MSRI) in 
Berkeley, California, for its hospitality and excellent working conditions. They would also like to thank Ben Antieau  for pointing out Gabber's work \cite{Gabber} as well for indicating an alternative proof to Theorem \ref{thm:main2}.


\section{Background on dg categories}\label{sec:dg}
Let $\cC(k)$ be the category of cochain complexes of $k$-modules; we use cohomological notation. A {\em differential graded (=dg) category $\cA$} is a category enriched over $\cC(k)$ (morphisms sets $\cA(x,y)$ are complexes) in such a way that composition fulfills the Leibniz rule $d(f \circ g) =d(f) \circ g +(-1)^{\mathrm{deg}(f)}f \circ d(g)$. A {\em dg functor} $F:\cA\to \cB$ is  a functor enriched over $\cC(k)$; consult Keller's ICM survey \cite{ICM-Keller}. In what follows we will write $\dgcat(k)$ for the category of (small) dg categories and dg ~functors. 
\subsection*{Modules}
Let $\cA$ be a dg category.  The category $\dgHo(\cA)$ has the same objects as $\cA$ and
morphisms given by $\dgHo(\cA)(x,y):=H^0(\cA(x,y))$, where $H^0$
denotes degree zero cohomology. The {\em opposite} dg category $\cA^\op$ has the same objects as $\cA$ and complexes of morphisms given by $\cA^\op(x,y):=\cA(y,x)$.  A {\em right $\cA$-module} is a dg functor $\cA^\op \to \cC_\dg(k)$ with values in the dg category $\cC_\dg(k)$ of cochain complexes of $k$-modules. Let us denote by $\cC(\cA)$ the category of right $\cA$-modules. Recall from \cite[\S3.2]{ICM-Keller} that the {\em derived category $\cD(\cA)$ of $\cA$} is the localization of $\cC(\cA)$ with respect to the class of objectwise quasi-isomorphisms. Its full subcategory of compact objects will be denoted by $\cD_c(\cA)$.
\subsection*{Morita equivalences}
A dg functor $F:\cA\to \cB$ is called a {\em Morita equivalence} if the induced restriction of scalars $\cD(\cB) \stackrel{\sim}{\to} \cD(\cA)$ is an equivalence of (triangulated) categories; see \cite[\S4.6]{ICM-Keller}. As proved in \cite[Thm.~5.3]{Additive}, $\dgcat(k)$ admits a Quillen model structure whose weak equivalences are precisely the Morita equivalences. Let us denote by $\Hmo(k)$ the homotopy category hence obtained.
\subsection*{Tensor product}
The {\em tensor product $\cA\otimes\cB$} of two dg categories $\cA$ and $\cB$ is defined by the cartesian product of the sets of objects of $\cA$ and $\cB$ and by the complexes of morphisms $(\cA\otimes\cB)((x,z),(y,w)):= \cA(x,y) \otimes \cB(z,w)$. As explained in \cite[\S2.3]{ICM-Keller}, this gives rise to a symmetric monoidal structure on $\dgcat(k)$ with $\otimes$-unit the dg category $\underline{k}$. After deriving it $-\otimes^{\bbL}-$, this symmetric monoidal structure descends to $\Hmo(k)$; consult \cite[\S4.3]{ICM-Keller} for details.
\subsection*{Bimodules}
Let $\cA, \cB \in \dgcat(k)$. A {\em $\cA\text{-}\cB$-bimodule $\mathsf{B}$} is a right $(\cA^\op \otimes \cB)$-module, \ie a dg functor $\mathsf{B}:\cA \otimes \cB^\op\to \cC_\dg(k)$. Standard examples are the $\cA\text{-}\cA$-bimodule
\begin{eqnarray}\label{eq:bimodule1}
\cA\otimes\cA^\op \to \cC_\dg(k) && (x,y) \mapsto \cA(y,x)
\end{eqnarray} 
as well as the $\cA\text{-}\cB$-bimodule 
\begin{eqnarray}\label{eq:bimodule2}
{}_F\mathsf{Bi}:\cA\otimes \cB^\op \to \cC_\dg(k) && (x,z) \mapsto \cB(z,F(x))
\end{eqnarray}
associated to a dg functor $F:\cA\to \cB$.
\subsection*{Smoothness and properness}
Recall from Kontsevich \cite{IAS,ENS,Miami,finMot} that a dg category $\cA$ is called {\em smooth} if the $\cA\text{-}\cA$-bimodule \eqref{eq:bimodule1} belongs to $\cD_c(\cA^\op\otimes^\bbL \cA)$ and {\em proper} if for each ordered pair of objects $(x,y)$ we have $\sum_i \mathrm{rank}\, H^i\cA(x,y)< \infty$.
\section{Background on noncommutative motives}\label{sec:NCmotives}
In this section we recall from \cite{Additive} the construction of the category of noncommutative motives; consult also the survey article \cite{survey}. Let $\cA, \cB \in \dgcat(k)$. As proved in
\cite[Cor.~5.10]{Additive}, one has a bijection 
\begin{equation}\label{eq:bij}
\Hom_{\Hmo(k)}(\cA,\cB)\simeq \mathrm{Iso}\,\rep(\cA,\cB)\,,
\end{equation}
where $\rep(\cA,\cB)$ denotes the full triangulated subcategory of
$\cD(\cA^\op \otimes^{\bbL} \cB)$ consisting of those
$\cA\text{-}\cB$-bimodules $\mathsf{B}$ such that for every $x \in \cA$ the right $\cB$-module $\mathsf{B}(x,-)$ belongs to
$\cD_c(\cB)$. Under \eqref{eq:bij}, the composition law in $\Hmo(k)$
corresponds to the (derived) tensor product of bimodules and the
identity of an object $\cA$ is given by the isomorphism class of the $\cA\text{-}\cA$-bimodule \eqref{eq:bimodule1}. Since the
$\cA\text{-}\cB$-bimodules \eqref{eq:bimodule2} belong to
$\rep(\cA,\cB)$, we hence obtain a well-defined symmetric monoidal functor
\begin{eqnarray}\label{eq:functor1}
\dgcat(k) \too \Hmo(k) && F \mapsto {}_F\mathsf{Bi}\,.
\end{eqnarray}
The {\em additivization} of $\Hmo(k)$ is the additive category $\Hmo_0(k)$ with the same objects as $\Hmo(k)$ and abelian groups of morphisms given by 
$$\Hom_{\Hmo_0(k)}(\cA,\cB):=K_0\,\rep(\cA,\cB)\,,$$
where $K_0\,\rep(\cA,\cB)$ is the Grothendieck group of the triangulated category $\rep(\cA,\cB)$. The composition law is induced by the (derived) tensor product of bimodules. Note that we have also a canonical functor
\begin{eqnarray}\label{eq:functor2}
\Hmo(k) \too \Hmo_0(k) && \mathsf{B} \mapsto [\mathsf{B}]\,.
\end{eqnarray}
Finally, given a commutative ring of coefficients $R$, the {\em $R$-linearization} of $\Hmo_0(k)$ is the $R$-linear additive category $\Hmo_0(k)_R$ obtained by tensoring each abelian group of morphisms of $\Hmo_0(k)$ with $R$. This gives rise to a functor\begin{eqnarray}\label{eq:functor3}
\Hmo_0(k) \too \Hmo_0(k)_R && [\mathsf{B}] \mapsto [\mathsf{B}]\otimes_\bbZ R\,.
\end{eqnarray}
As proved in \cite{Additive}, the symmetric monoidal structure on $\Hmo(k)$ descends first to a bilinear symmetric monoidal structure on $\Hmo_0(k)$ and then to a $R$-linear bilinear symmetric monoidal structure on $\Hmo_0(k)_R$, making \eqref{eq:functor2}\text{-}\eqref{eq:functor3} into symmetric monoidal functors. The universal additive invariant with $R$-coefficients \eqref{eq:universal} is then defined by the following composition
$$ U(-)_R:\dgcat(k) \stackrel{\eqref{eq:functor1}}{\too} \Hmo(k) \stackrel{\eqref{eq:functor2}}{\too} \Hmo_0(k)\stackrel{\eqref{eq:functor3}}{\too} \Hmo_0(k)_R\,.$$
Finally, given dg categories $\cA, \cB \in \dgcat(k)$ with $\cA$ smooth and proper, the triangulated category $\rep(\cA,\cB) \subset \cD(\cA^\op \otimes^\bbL \cB)$ identifies with $\cD_c(\cA^\op \otimes^\bbL \cB)$; see \cite[\S5]{CT1}. As a consequence, we obtain the following isomorphism
\begin{equation}\label{eq:smooth-case}
 \Hom_{\Hmo_0(k)_R}(U(\cA)_R, U(\cB)_R) \simeq K_0(\cA\otimes^\bbL \cB)_R\,.
 \end{equation}
\section{Perfect complexes}\label{sec:schemes}
Let $k$ be a base commutative ring, $X$ a quasi-compact
quasi-separated $k$-scheme, and $A$ a sheaf of $\cO_X$-algebras. We
introduce some notations and concepts that are standard in the
particular case where $A=\cO_X$ (consult
\cite[\S3]{BV}\cite[\S4.4]{ICM-Keller} and the references therein) and
whose generalization to an arbitrary sheaf $A$ is immediate.

Let $\Mod(A)$ be the Grothendieck category of
sheaves of (right) $A$-modules, $\mathrm{Qcoh}(A)$ the full subcategory
of quasi-coherent $A$-modules, $\cD(A):=D(\Mod(A))$ the derived category of $A$, and $\cD_{\mathrm{Qcoh}}(A) \subset \cD(A)$ the full triangulated subcategory of those complexes of $A$-modules with quasi-coherent cohomology. When $X$ is separated we have $\cD_{\mathrm{Qcoh}}(A)\simeq D(\mathrm{Qcoh}(A))$.
\begin{definition}\label{def:perfect}
A complex of $A$-modules $\cF \in \cD(A)$ is called {\em perfect} if there exists a covering $X=\bigcup_i V_i$ of $X$ by affine open subschemes $V_i \subset X$ such that for every $i$ the restriction $\cF_{|V_i}$ of $\cF$ to $V_i$ is quasi-isomorphic to a bounded complex of finitely generated projective $A_{|V_i}$-modules. Let us denote by $\perf(A)$ the triangulated category of perfect complexes. Note that by construction we have the inclusions $\perf(A) \subset \cD_{\mathrm{Qcoh}}(A)\subset \cD(A)$.
\end{definition}
Let $\cE$ be an abelian (or exact) category. As explained in \cite[\S4.4]{ICM-Keller}, the derived dg category $\cD_\dg(\cE)$ of $\cE$ is defined as the dg quotient $\cC_\dg(\cE)/\cA c_\dg(\cE)$ of the dg category of complexes over $\cE$ by its full dg subcategory of acyclic complexes. Note that every exact functor $\cE \to \cE'$ (or more generally every dg functor $\cC_\dg(\cE) \to \cC_\dg(\cE')$ which restricts to $\cA c_\dg(\cE) \to \cA c_\dg(\cE')$) gives rise to a dg functor $\cD_\dg(\cE) \to \cD_\dg(\cE')$.
\begin{notation}
Let us write $\cD_\dg(A)$ for the dg category $\cD_\dg(\cE)$ with $\cE:=\Mod(A)$, $\cD_{\mathrm{Qcoh},\dg}(A)$ for the full dg category of those complexes of $A$-modules with quasi-coherent cohomology, and $\perf_\dg(A)$ for the full dg subcategory of perfect complexes. By construction, we have inclusions $\perf_\dg(A) \subset \cD_{\mathrm{Qcoh},\dg}(A) \subset \cD_\dg(A)$ of dg categories and canonical equivalences
\begin{eqnarray*}
\dgHo(\perf_\dg(A)) \simeq \perf(A) & \dgHo(\cD_{\mathrm{Qcoh},\dg}(A)) \simeq \cD_{\mathrm{Qcoh}}(A) & \dgHo(\cD_\dg(A)) \simeq \cD(A)\,.
\end{eqnarray*}
When $A=\cO_X$ we will write in what follows $X$ instead of $\cO_X$.
\end{notation}
Note that we have a well-defined forgetful functor $\cD(A) \to \cD(X)$ which restricts to $\perf(A) \to \perf(X)$ when $A$ is perfect as a complex of $\cO_X$-modules. Since the forgetful functor $\Mod(A) \to \Mod(X)$ is exact, one has similar forgetful dg functors $\cD_\dg(A) \to \cD_\dg(X)$ and $\perf_\dg(A) \to \perf_\dg(X)$. The following (well-known) fact will be used in the sequel.
\begin{lemma}
\label{lemma:wellknown}
Let $X$ and $A$ be as above, with $A$ a sheaf of Azumaya algebras over $X$. Under these assumptions, the following square is cartesian
$$
\xymatrix{
\perf(A) \ar[d]_-{\mathrm{forget}} \ar[r] \ar@{}[dr]|-{\ulcorner} & \cD(A) \ar[d]^-{\mathrm{forget}}  \\
\perf(X) \ar[r] & \cD(X)\,,
}
$$
\ie a complex $\cF \in \cD(A)$ belongs to $\perf(A)$ if an only if it belongs to $\perf(X)$.
\end{lemma}
\begin{proof} 
Thanks to the above Definition \ref{def:perfect}, it suffices to prove the affine case where $X=\mathrm{Spec}(S)$ and $A$ is an Azumaya algebra over $S$. Recall from \cite[III \S5]{Knus} that:
\begin{itemize}
\item[(i)] $A$ is finitely generated and projective as a right $S$-module;
\item[(ii)] $A$ is {\em separable}, \ie $A$ is projective as a $A\text{-}A$-bimodule.  
\end{itemize}
If by hypothesis $\cF$ belongs to $\perf(A)$, then condition (i) allows us to conclude that $\cF$ also belongs to $\perf(X)$. In order to prove the converse implication, consider the base-change functor $-\otimes_SA:\cD(S) \to \cD(A)$. By construction, it preserves perfect complexes. Hence, if by hypothesis $\cF$ belongs to $\perf(X)$, $\cF\otimes_SA$ belongs to $\perf(A)$. Now, consider the following short exact sequence of $A\text{-}A$-bimodules
\begin{equation}\label{eq:short}
0 \too \mathrm{Ker}(m) \too A \otimes_S A \stackrel{m}{\too} A \too 0\,,
\end{equation}
where $m$ stands for the multiplication of $A$. Thanks to the above condition (ii), \eqref{eq:short} splits and hence $A$ becomes a direct summand of the $A\text{-}A$-bimodule $A\otimes_SA$. Using the canonical isomorphism $\cF \otimes_SA \simeq \cF\otimes_A (A\otimes_SA)$ one concludes that $\cF$ is a direct summand of $\cF\otimes_SA$. Since $\cF\otimes_SA$ belongs to $\perf(A)$ and this category is idempotent complete, $\cF$ also belongs to $\perf(A)$. This completes the proof.

\end{proof}
Every sheaf $A$ of $\cO_X$-algebras gives rise to the following dg functor
\begin{eqnarray}\label{eq:dg-functor}
-\otimes_{\cO_X}^\bbL A: \cC_\dg(\Mod(X)) \too \cC_\dg(\Mod(A)) && \cF \mapsto \cF_{\mathrm{flat}}\otimes_{\cO_X} A\,, 
\end{eqnarray}
where $\cF_{\mathrm{flat}}$ denotes a (functorial) $\cO_X$-flat resolution of $\cF$. Note that when $A$ is $\cO_X$-flat (\eg\ $A$ locally free of finite rank over $\cO_X$), \eqref{eq:dg-functor} identifies with $-\otimes_{\cO_X}A$. Since \eqref{eq:dg-functor} preserves acyclic and perfect complexes, it induces a dg functor  
$$ - \otimes_{\cO_X}^\bbL A: \perf_\dg(X) \too \perf_\dg(A)\,.$$

\section{Proof of Theorem~\ref{thm:main2}}
Since $X$ is quasi-compact and quasi-separated, the proof can be reduced to the affine case and to a Mayer-Vietoris argument; see \cite[Prop.~3.3.1]{BV}.
\begin{lemma} (Thomason-Trobaugh \cite[\S8]{TT})\label{lem:sequence}
  Let $X$ be a quasi-compact quasi-separated scheme, and $U_1,U_2$ two Zariski open subschemes. Assume that $X=U_1 \cup U_2$ and write $U_{12}:= U_1 \cap U_2$. Under these assumptions, one has the exact sequence.
\begin{equation}\label{eq:long-seq}
 K_1(U_1) \oplus K_1(U_2) \to K_1(U_{12}) \stackrel{\partial}{\to} K_0(X)  \stackrel{\pm}{\to} K_0(U_1) \oplus K_0(U_2) \to K_0(U_{12})\,.
\end{equation}
\end{lemma}
\begin{proof}
Let us write $\iota_1: U_1 \hookrightarrow X$ and $\iota_2:U_2 \hookrightarrow X$ for the two open inclusions. Consider the following commutative diagram in $\Hmo(k)$
\begin{equation}\label{eq:sequences}
\xymatrix@C=2em@R=2em{
0 \ar[r] & \perf_\dg(X)_Z \ar[d]_-\sim \ar[r] & \perf_\dg(X) \ar[d]_-{\bbL\iota_2^\ast} \ar[r]^-{\bbL\iota_1^\ast} & \perf_\dg(U_1) \ar[d] \ar[r] & 0  \\
0 \ar[r] & \perf_\dg(U_2)_{Z'} \ar[r] &\perf_\dg(U_2) \ar[r] & \perf_\dg(U_{12}) \ar[r] & 0\,,
}
\end{equation}
where $Z$ (resp. $Z'$) is the closed set $X- U_1$ (resp. $U_2- U_{12}$) and $\perf_\dg(X)_Z$ (resp. $\perf_\dg(U_2)_{Z'}$) the dg category of those perfect complexes of $\cO_X$-modules (resp. of $\cO_{U_2}$-modules) that are supported on $Z$ (resp. on $Z'$). Recall from \cite[\S4.6]{ICM-Keller} the notion of a short exact sequence of dg categories. Roughly speaking, it consists of a sequence of dg categories $\cA \to \cB \to \cC$ for which $\cD(\cA) \to \cD(\cB) \to \cD(\cC)$ is exact in the sense of Verdier. As explained in \cite[\S5]{TT}, both rows in \eqref{eq:sequences} are short exact sequences of dg categories; see also \cite[\S4.6]{ICM-Keller}. Furthermore, as proved in \cite[Thm.~2.6.3]{TT}, the induced dg functor $\perf_\dg(X)_Z \stackrel{\sim}{\to} \perf_\dg(U_2)_{Z'}$ is a Morita equivalence and hence an isomorphism in $\Hmo(k)$. 

Nonconnective algebraic $K$-theory gives rise to a functor $\bbK:\Hmo(k) \to \Ho(\Spt)$ with values in the homotopy category of spectra. Among other properties, it sends short exact sequences of dg categories to distinguished triangles of spectra; see \cite{Negative}\cite[Thm.~10.9]{Higher}. Hence, by applying it to \eqref{eq:sequences} we obtain the following morphism between distinguished triangles
$$
\xymatrix{
\bbK\perf_\dg(X)_Z \ar[d]_-{\sim} \ar[r] & \bbK\perf_\dg(X) \ar[d]_-{\bbK(\bbL\iota_2^\ast)} \ar[r]^-{\bbK(\bbL\iota^\ast_1)} & \bbK\perf_\dg(U_1) \ar[d] \ar[r] & \Sigma \bbK\perf_\dg(X)_Z \ar[d]^-{\sim} \\
\bbK\perf_\dg(U_2)_{Z'} \ar[r] & \bbK\perf_\dg(U_2) \ar[r] & \bbK\perf_\dg(U_{12}) \ar[r] & \Sigma \bbK\perf_\dg(U_2)_{Z'}\,.
}
$$
Since the outer left and right vertical maps are isomorphisms we hence obtain a Mayer-Vietoris long exact sequence
\begin{equation*}
\cdots \to  K_{n+1}(U_1) \oplus K_{n+1}(U_2) \to K_{n+1}(U_{12}) \stackrel{\partial}{\to} K_n(X)  \stackrel{\pm}{\to} K_n(U_1) \oplus K_n(U_2) \to \cdots \,,
\end{equation*}
where the boundary maps $\partial$'s are obtained from the composition
\begin{equation}
\label{eq:partial}
\bbK\perf_\dg(U_{12})\too \Sigma \bbK\perf_\dg(U_2)_{Z'}
\overset{\sim}{\too} \Sigma \bbK\perf_\dg(X)_Z
\end{equation}
The exact sequence \eqref{eq:long-seq} is a chunk of the above one and so the proof is finished.
\end{proof}
Recall now from \cite[page~323]{Weibel} that the pairing
$$ - \otimes^\bbL_{\cO_X}- : \perf(X) \times \perf(X) \too \perf(X)$$
endows $K_\ast(X)$ with a graded-commutative ring structure.
\begin{lemma} 
\label{lem:MV}
Let $X$, $U_1$, $U_2$, $U_{12}$ be as in the above Lemma~\ref{lem:sequence}. Given an element $\alpha$ in $K_0(X)$, we have the following commutative diagram
$$
\xymatrix@C=1.5em@R=3em{
K_1(U_1) \oplus K_1(U_2) \ar[d]_-{-\cdot \alpha_1 \oplus - \cdot \alpha_2} \ar[r] & K_1(U_{12}) \ar[d]_-{-\cdot \alpha_{12}} \ar[r]^-{\partial} & K_0(X) \ar[d]_-{-\cdot \alpha} \ar[r]^-\pm & K_0(U_1) \oplus K_0(U_2) \ar[d]_-{-\cdot \alpha_1 \oplus - \cdot \alpha_2} \ar[r] & K_0(U_{12}) \ar[d]_-{-\cdot \alpha_{12}} \\
K_1(U_2) \oplus K_1(U_2) \ar[r] & K_1(U_{12}) \ar[r]_-\partial & K_0(X) \ar[r]_-\pm& K_0(U_1) \oplus K_0(U_2) \ar[r] & K_0(U_{12})\,,
}
$$
where $\alpha_1$, $\alpha_2$ and $\alpha_{12}$ denote the images of $\alpha$ in $K_0(U_1)$, $K_0(U_2)$ and $K_0(U_{12})$, respectively.
\end{lemma}
\begin{proof}
Recall first from Thomason \cite[\S1.6]{Thomason} that every element $\alpha$ in $K_0(X)$ is of the form $\alpha=[\cF]$ for some perfect complex $\cF$. Note that we have the following commutative cube in the homotopy category $\Hmo(k)$
\begin{equation}\label{eq:com-diagram}
\xymatrix{
\perf_\dg(X) \ar[ddd]_-{\bbL\iota^\ast_2} \ar[rrr]^-{\bbL\iota^\ast_1} & & & \perf_\dg(U_1) \ar[ddd] \\
& \perf_\dg(X) \ar[ul]_-{-\otimes^\bbL_{\cO_X} \cF} \ar[r]^-{\bbL\iota_1^\ast} \ar[d]_-{\bbL\iota_2^\ast} & \perf_\dg(U_1) \ar[d] \ar[ur]_-{-\otimes^\bbL_{\cO_X} \cF_1} & \\
& \perf_\dg(U_2) \ar[dl]^-{-\otimes^\bbL_{\cO_X} \cF_2} \ar[r] & \perf_\dg(U_{12}) \ar[dr]^-{- \otimes^\bbL_{\cO_X} \cF_{12}} & \\
\perf_\dg(U_2) \ar[rrr] &&& \perf_\dg(U_{12}) \,,
}
\end{equation}
where $\cF_1, \cF_2$ and $\cF_{12}$ denote the restriction of $\cF$ to $U_1$, $U_2$ and $U_{12}$, respectively. Following the proof of Lemma~\ref{lem:sequence}, one observes that the commutative cube \eqref{eq:com-diagram} gives rise to a morphism between Mayer-Vietoris long exact sequences
$$
\xymatrix@C=1.5em@R=3em{
\cdots K_{n+1}(U_1) \oplus K_{n+1}(U_2) \ar[d]_-{-\cdot \alpha_1 \oplus - \cdot \alpha_2} \ar[r] & K_{n+1}(U_{12}) \ar[d]_-{-\cdot \alpha_{12}} \ar[r]^-{\partial} & K_n(X) \ar[d]_-{-\cdot \alpha} \ar[r]^-\pm & K_n(U_1) \oplus K_n(U_2) \ar[d]_-{-\cdot \alpha_1 \oplus - \cdot \alpha_2} \cdots  \\
\cdots K_{n+1}(U_2) \oplus K_{n+1}(U_2) \ar[r] & K_{n+1}(U_{12}) \ar[r]_-\partial & K_n(X) \ar[r]_-\pm& K_n(U_1) \oplus K_n(U_2) \cdots\,,
}
$$
where the commutativity of the middle square follows from composition \eqref{eq:partial}.
The diagram of Lemma~\ref{lem:MV} is a chunk of this one and so the proof is finished.
\end{proof}
\begin{lemma}\label{prop:MV}
Let $X, U_1, U_2, U_{12}$ and $\alpha, \alpha_1, \alpha_2, \alpha_{12}$ be as in Lemma \ref{lem:MV}. Whenever $\alpha_1$, $\alpha_2$ and $\alpha_{12}$ are nilpotent, $\alpha$ is also nilpotent.
\end{lemma} 
\begin{proof}
Since by hypothesis $\alpha_1, \alpha_2$ and $\alpha_{12}$ are nilpotent, there exists an integer $N \gg 0$ for which the homomorphisms
\begin{eqnarray*}
K_0(U_1) \stackrel{-\cdot\alpha_1^N}{\too} K_0(U_1) &K_0(U_2) \stackrel{-\cdot\alpha_2^N}{\too} K_0(U_2) & K_1(U_{12}) \stackrel{- \cdot \alpha_{12}^N}{\too} K_1(U_{12}) 
\end{eqnarray*}
are all trivial. Consequently, using the above Lemma~\ref{lem:MV} (with $\alpha$ replaced by $\alpha^N$), one obtains the following commutative diagram
$$
\xymatrix{
K_1(U_{12}) \ar[d]_-{-\cdot \alpha_{12}^N} \ar[r]^-\partial 
& K_0(X) \ar[d]_-{-\cdot \alpha^N} \ar[r]^-\pm & K_0(U_1) \oplus K_0(U_2) \ar[d]_-{-\cdot \alpha_1^N \oplus -\cdot \alpha_2^N}^{\textbf{0}} 
\\
K_1(U_{12}) \ar[d]_-{-\cdot \alpha_{12}^N}^{\textbf{0}} \ar[r]^-\partial & K_0(X) \ar[d]_-{-\cdot \alpha^N} \ar[r]^-\pm & K_0(U_1) \oplus K_0(U_2) \ar[d]_-{-\cdot \alpha_1^N \oplus -\cdot \alpha_2^N} \\
K_1(U_{12}) \ar[r]_-\partial & K_0(X)  \ar[r]_-\pm & K_0(U_1) \oplus K_0(U_2)\,.
}
$$
A simple diagram chasing argument then shows that the composition of the middle vertical arrows is zero.
I.e.\ $-\cdot \alpha^{2N}=0$. This implies that $\alpha^{2N}=0$ and hence that $\alpha$ is nilpotent.
\end{proof}
We now have all the ingredients needed for the proof of Theorem~\ref{thm:main2}. Let $\alpha$ be an element in $I_X \subset K_0(X)$. One needs to show that $\alpha$ is nilpotent. This will be done in three steps.
\subsection*{Step 1}
We claim that $X=\bigcup_{i=1}^n V_i$, where $V_i\subset X$ is an affine open subscheme such that the image 
of $\alpha$ in $K_0(V_i)$ is zero. Let $x\in X$ and $V_x$ an affine open neighborhood of $x$. The image $\alpha_{\mid V_x}$ of $\alpha$ in $K_0(V_x)$ can be written as $[P]-[Q]$ with $P$ and $Q$ two vector bundles of the same rank. By shrinking $V_x$ we can assume
that $P$ and $Q$ are free of the same rank and hence isomorphic. As a consequence, $\alpha_{\mid V_x}=0$. Finally, using quasi-compactness, we may take a finite subcover $\{V_i\}_{i=1}^n$ of $\{V_x\}_{x \in X}$
 which yields the above claim.

\subsection*{Step 2} Assume that $X$ is a quasi-compact {\em separated} scheme. We prove Theorem~\ref{thm:main2} using induction on the number of affine open subschemes in a covering trivializing $\alpha$ as in Step 1. The case $n=1$ is clear. Let us then assume that $n>1$ and write $U_1:=\cup_{i=1}^{n-1}V_i$, $U_2:=V_n$ and $U_{12}:=U_1\cap U_2$. Since by hypothesis $X$ is separated, $V_i\cap V_j$ is affine for all $i,j$ and so $U_1$ and $U_{12}$ are covered by $n-1$
affine open subschemes on which the restriction of $\alpha$ is trivial. By our induction hypothesis, $\alpha_1$, $\alpha_2$, $\alpha_{12}$ are nilpotent. Hence, using the above Lemma
\ref{prop:MV}, we conclude that $\alpha$ is also nilpotent.

\subsection*{Step 3} Assume that $X$ is a quasi-compact quasi-separated scheme. Let $U_1$,
$U_2$ and $U_{12}$ be as in Step 2. Note that $U_1$ is covered by $n-1$
affine open subschemes on which $\alpha$ is trivial and that $U_2$ and $U_{12}$ are
separated (since $U_2$ is affine and $U_{12}\subset U_2$). Therefore, using induction and Step 2, we can again assume that $\alpha_1$, $\alpha_2$,
$\alpha_{12}$ are nilpotent. We finish the proof by invoking Lemma \ref{prop:MV} once again.
\subsection*{Proof of Corollary~\ref{cor:main2}}
Every ring homomorphism $R \to R'$ gives rise to a well-defined ring homomorphism $K_0(X)_R \to K_0(X)_{R'}$. Hence, since $\bbZ[1/r]$ is initial among the rings containing $1/r$, it suffices to prove the particular case $R:=\bbZ[1/r]$. Note that since $\bbZ[1/r]$ is torsion-free we have the following short exact sequence
$$ 0 \to I_X \otimes \bbZ[1/r] \subset K_0(X)_{\bbZ[1/r]} \stackrel{\mathrm{rank}}{\twoheadrightarrow} \bbZ^m \otimes \bbZ[1/r] \to 0\,.$$
Moreover, thanks to Theorem~\ref{thm:main2}, every element in $I_X \otimes \bbZ[1/r]$ is nilpotent. The rank homomorphism is surjective and so there exists an element $\beta \in K_0(X)_{\bbZ[1/r]}$ of rank $(1/r_1, \ldots, 1/r_m)$. Therefore, $\alpha \cdot \beta$ is of rank $(1, \ldots, 1)$ and consequently $([\cO_X]-\alpha \cdot \beta) \in I_X \otimes \bbZ[1/r]$. There exists then an integer $N\gg 0$ such that $([\cO_X]-\alpha \cdot \beta)^{N+1}=0$. This implies that the following element
$$ [\cO_X] + ([\cO_X] - \alpha \cdot \beta) + ([\cO_X] - \alpha \cdot \beta)^2 + \cdots + ([\cO_X] - \alpha \cdot \beta)^N \in K_0(X)_{\bbZ[1/r]}$$
is the inverse of $\alpha \cdot \beta$ and hence that $\alpha$ is invertible in $K_0(X)_{\bbZ[1/r]}$.
\section{Proof of Theorem~\ref{thm:main}}\label{sec:proof-main}
The proof is divided into two steps. First, we introduce an auxiliary
$\bbZ[1/r]$-linear category $\Az_0(X)_{1/r}$ of sheaves of Azumaya algebras over $X$ and prove
the analogue of Theorem~\ref{thm:main} therein; see
Proposition~\ref{prop:key}. Then, we construct a $\bbZ[1/r]$-linear functor
from $\Az_0(X)_{1/r}$ to the category $\Hmo_0(k)_{\bbZ[1/r]}$ of noncommutative motives. 

Note that every ring homomorphism $R\to R'$ gives rise to a
well-defined additive functor $\Hmo_0(k)_R \to \Hmo_0(k)_{R'}$. Hence,
since $\bbZ[1/r]$ is initial among the rings containing $1/r$, it
suffices to prove the case $R=\bbZ[1/r]$.
\subsection*{Auxiliary category $\Az_0(X)$}
  Given two sheaves $A$ and $B$ of Azumaya algebras over $X$, let $\rep(A,B)$ be the full triangulated subcategory of
  $\cD(A^\op\otimes_{\cO_X}B)$ consisting of those
  $A\text{-}B$-bimodules ${}_A\mathsf{B}_B$ such that $\mathsf{B}_B
  \in \perf(B)$.
\begin{lemma}\label{lem:rep=perf}
The categories $\rep(A,B)$ and $\perf(A^\op\otimes_{\cO_X}B)$ are the same.
\end{lemma}
\begin{proof}
We start with the inclusion $\rep(A,B) \subseteq \perf(A^\op \otimes_{\cO_X}B)$. Let ${}_A \mathsf{B}_B$ be an object of $\rep(A,B) \subset \cD(A^\op \otimes_{\cO_X} B)$. By definition, $\mathsf{B}_B \in \perf(B)$. Hence, Lemma \ref{lemma:wellknown} (with $A=B$) shows us that $B \in \perf(X)$. Using again Lemma \ref{lemma:wellknown} (with $A=A^\op \otimes_{\cO_X} B$) we conclude that ${}_A \mathsf{B}_B \in \perf(A^\op \otimes_{\cO_X}B)$.  

We now show the converse inclusion. Let ${}_A\mathsf{B}_B$ be an object of $\perf(A^\op \otimes_{\cO_X} B) \subset \cD(A^\op \otimes_{\cO_X} B)$. Lemma \ref{lemma:wellknown} (with $A=A^\op \otimes_{\cO_X}B$) shows us that $\mathsf{B} \in \perf(X)$. Using again Lemma \ref{lemma:wellknown} (with $A=B$) we conclude that $\mathsf{B}_B \in \perf(B)$. By definition, this implies that ${}_A \mathsf{B}_B \in \rep(A,B)$ and so the proof if finished.
\end{proof}
Let $\Az(X)$ be the category whose objects are the sheaves of Azumaya algebras over $X$, whose morphisms are given by $\Hom_{\Az(X)}(A,B):=\mathrm{Iso}\, \rep(A,B)$, and whose composition law is induced by
\begin{eqnarray}\label{eq:bifunctor-rep}
\rep(A,B) \times \rep(B,C) \too \rep(A,C) && ({}_A\mathsf{B}_B, {}_B \mathsf{B}'_C) \mapsto {}_A \mathsf{B} \otimes^\bbL_B \mathsf{B}'_C\,.
\end{eqnarray}
Note that the identity of an object $A \in \Az(X)$ is given by the isomorphism class of the $A\text{-}A$-bimodule ${}_AA_A$. The {\em additivization} of $\Az(X)$ is the additive category $\Az_0(X)$ with the same objects as $\Az(X)$ and with abelian groups of morphisms given by $\Hom_{\Az_0(X)}(A,B):=K_0\,\rep(A,B)$, where $K_0\,\rep(A,B)$ is the Grothendieck group of the triangulated category $\rep(A,B)$. The composition law is induced by the above bi-triangulated functor \eqref{eq:bifunctor-rep}. Note that we have a functor 
\begin{eqnarray*}
\Az(X) \to \Az_0(X) && {}_A\mathsf{B}_B \mapsto [{}_A\mathsf{B}_B]
\end{eqnarray*}
Finally, the {\em $\bbZ[1/r]$-linearization} of $\Az_0(X)$ is the $\bbZ[1/r]$-linear category $\Az_0(X)_{1/r}$ obtained by tensoring each abelian group of morphisms of $\Az_0(X)$ with $\bbZ[1/r]$. This gives rise to the functor 
\begin{eqnarray*}
\Az_0(X) \to \Az_0(X)_{1/r} && [{}_A\mathsf{B}_B] \mapsto [{}_A\mathsf{B}_B]_{1/r}:= [{}_A\mathsf{B}_B] \otimes_\bbZ \bbZ[1/r]\,.
\end{eqnarray*}
\begin{proposition}\label{prop:key}
Let $X, A$ be as in Theorem~\ref{thm:main}. Under these assumptions and the above notations, one has the isomorphism $[{}_{\cO_X}A_A]_{1/r}:\cO_X \stackrel{\sim}{\to} A$ in $\Az_0(X)_{1/r}$.
\end{proposition}
\begin{proof}
By definition, $A$ is locally free of finite rank over $\cO_X$. Consequently, the $A\text{-}\cO_X$-bimodule ${}_A A_{\cO_X}$ belongs to $\rep(A,\cO_X)$ and so one obtains a well-defined morphism $[{}_AA_{\cO_X}]_{1/r}:A \to \cO_X$ in $\Az_0(X)_{1/r}$. The proof will consist in showing that both compositions
\begin{eqnarray}\label{eq:compositions}
[{}_{\cO_X}A_A]_{1/r} \circ [{}_AA_{\cO_X}]_{1/r} && [{}_A A_{\cO_X}]_{1/r} \circ [{}_{\cO_X} A_A]_{1/r}
\end{eqnarray}
are isomorphisms. Thanks to the above Lemma~\ref{lem:rep=perf} (with $A=B=\cO_X$), one has the following $\bbZ[1/r]$-algebra isomorphism
\begin{equation}\label{eq:End-1}
\End_{\Az_0(X)_{1/r}}(\cO_X) := K_0(\rep(\cO_X,\cO_X))_{1/r} \simeq K_0(\perf(X))_{1/r} =: K_0(X)_{1/r}\,,
\end{equation}
where the right-hand-side is endowed with the multiplication induced by $-\otimes^\bbL_{\cO_X}-$. Since $A\otimes_AA\simeq A$, the composition $[{}_{\cO_X}A_A]_{1/r} \circ [{}_AA_{\cO_X}]_{1/r}$ equals $[{}_{\cO_X}A_{\cO_X}]_{1/r}$. Hence, since by hypothesis $A$ is of rank $(r_1, \ldots, r_m)$, we conclude from Corollary~\ref{cor:main2} and from Isomorphism \eqref{eq:End-1} that $[{}_{\cO_X}A_{\cO_X}]_{1/r}$ is invertible in $\End_{\Az_0(X)_{1/r}}(\cO_X)$. The first composition in \eqref{eq:compositions} is then an isomorphism. 

Let us now prove that the second composition in \eqref{eq:compositions} is also an isomorphism. Thanks to Lemma~\ref{lem:rep=perf} (with $A=B$), one has the $\bbZ[1/r]$-algebra isomorphism
\begin{equation}\label{eq:End-2}
\End_{\Az_0(X)_{1/r}}(A) := K_0(\rep(A,A))_{1/r} \simeq K_0(\perf(A^\op \otimes_{\cO_X}A))_{1/r}\,,
\end{equation}
where the right-hand-side is endowed with the multiplication induced by $-\otimes^\bbL_A-$. On the other hand, Lemma~\ref{lem:key} below furnishes us the following ring isomorphism
\begin{eqnarray}\label{eq:iso}
K_0(X) \stackrel{\sim}{\too} K_0(A^\op \otimes_{\cO_X}A) && \cF \mapsto \cF \otimes^\bbL_{\cO_X} A\,.
\end{eqnarray}
Now, note that the composition $[{}_A A_{\cO_X}]_{1/r} \circ
[{}_{\cO_X} A_A]_{1/r}$ is equal to $[{}_AA \otimes_{\cO_X}A_A]_{1/r}$. There exists then a unique element $\alpha$ in $K_0(X)$ which is mapped to $[{}_AA\otimes_{\cO_X} A_A]$ via the above isomorphism \eqref{eq:iso}. We claim that $\mathrm{rank}(\alpha)=(r_1, \ldots,r_m)$. In order to prove this claim, consider the composed functor
\begin{equation}\label{eq:composed}
\perf(X) \stackrel{\eqref{eq:tensor-eq}}{\too} \perf(A^\op \otimes_{\cO_X}A) \stackrel{\mathrm{forget}}{\too} \perf(X)\,.
\end{equation}
Since the $\cO_X$-rank of $A$ is $(r_1, \ldots, r_m)$, \eqref{eq:composed} gives rise to the commutative square
\begin{equation}\label{eq:rank-diagram}
\xymatrix{
K_0(X) \ar@{>>}[d]_-{\mathrm{rank}} \ar[rr]^-{-\cdot [A]} && K_0(X) \ar@{>>}[d]^-{\mathrm{rank}} \\
\bbZ^m \ar[rr]_-{- \cdot (r_1, \ldots r_m)} && \bbZ^m\,.
}
\end{equation}
The equalities $\mathrm{rank}(\alpha \cdot [A]) = \mathrm{rank}([A\otimes_{\cO_X}A]) = (r_1, \ldots, r_m)^2$, combined with the commutativity of \eqref{eq:rank-diagram} and the injectivity of the homomorphism $- \cdot (r_1, \ldots, r_m)$, allows us then to conclude that $\mathrm{rank}(\alpha)=(r_1, \ldots, r_m)$. Thanks to Corollary \ref{cor:main2}, the element $\alpha$ becomes then invertible in $K_0(X)_{1/r}$ and so using \eqref{eq:End-2} and the $\bbZ[1/r]$-linearization of \eqref{eq:iso}, one concludes that $[{}_AA \otimes_{\cO_X}A_A]_{1/r}$ is invertible in $\End_{\Az_0(X)_{1/r}}(A)$. This implies that the second composition in \eqref{eq:compositions} is also an isomorphism, and so the proof is finished.
\end{proof}
\begin{lemma}\label{lem:key}
Let $X, A$ be as in Theorem~\ref{thm:main}. Under these assumptions, one has the following equivalence of monoidal triangulated categories
\begin{eqnarray}\label{eq:tensor-eq}
\perf(X) \stackrel{\sim}{\too} \perf(A^\op \otimes_{\cO_X} A) && \cF \mapsto \cF \otimes^\bbL_{\cO_X} A\,,
\end{eqnarray}
where the monoidal structure on $\perf(X)$ (resp. on $\perf(A^\op \otimes_{\cO_X}A)$) is induced by $-\otimes^\bbL_{\cO_X}-$ (resp. by $-\otimes^\bbL_A-$).
\end{lemma}
\begin{remark}
Since the monoidal structure on $\perf(X)$ is symmetric, we conclude from \eqref{eq:tensor-eq} that the monoidal structure on $\perf(A^\op \otimes_{\cO_X} A)$ is also symmetric.
\end{remark}
\begin{proof}
The fact that \eqref{eq:tensor-eq} is monoidal follows from the canonical isomorphisms 
$$(\cF\otimes^\bbL_{\cO_X}A) \otimes^\bbL_A(\cF' \otimes^\bbL_{\cO_X}A) \simeq (\cF \otimes^\bbL_{\cO_X} \cF') \otimes^\bbL_{\cO_X}A\,.$$
In order to prove that \eqref{eq:tensor-eq} is moreover an equivalence it suffices from Definition \ref{def:perfect} to show the affine case where $X=\mathrm{Spec}(S)$ and $A$ is an Azumaya algebra over $S$. As explained in \cite[\S III Thm.~5.1 3)]{Knus}, one has an equivalence of categories
\begin{eqnarray}\label{eq:equivalence}
\Mod(S) \stackrel{\sim}{\too} \Mod(A^\op \otimes_S A) && \cF \mapsto \cF \otimes_S A\,.
\end{eqnarray}
Since \eqref{eq:equivalence} preserves finitely generated projective modules, one concludes from the definition of perfect complex that \eqref{eq:tensor-eq} is an equivalence in the affine case. This completes the proof.
%
%
%
%
\end{proof}
\begin{remark}
  Using Proposition~\ref{prop:key}, one observes that the category
  $\Az_0(X)_\bbQ$ (obtained by tensoring each abelian group of morphisms of $\Az_0(X)$ with $\bbQ$) has a single isomorphism class.
Kontsevich calls such categories ``algebroids'' \cite[\S1.1]{Ko10}.
Intuitively speaking, all the complexity of $\Az_0(X)$ is torsion. 
\end{remark}
\subsection*{From $\Az_0(X)$ to noncommutative motives}
Let $A, B \in \Az_0(X)$. Note that every $A\text{-}B$-bimodule ${}_A\mathsf{B}_B \in \rep(A,B)$ gives rise to a dg functor 
\begin{eqnarray*}
- \otimes^\bbL_A \mathsf{B}: \perf_\dg(A) \too \perf_\dg(B) && \cF \mapsto \cF_{\mathrm{flat}}\otimes_A \mathsf{B}
\end{eqnarray*}
and consequently to a bimodule ${}_{(-\otimes^\bbL_A \mathsf{B})}\mathsf{Bi}$ which belongs to $\rep(\perf_\dg(A),\perf_\dg(B))$; recall from \eqref{eq:bimodule2} the notation ${}_-\mathsf{Bi}$. Similarly, every morphism $f: {}_A
  \mathsf{B}_B \to {}_A \mathsf{B}'_B$ of $A\text{-}B$-bimodules gives
  rise to a morphism of dg functors $\nu_f: -\otimes^\bbL_A \mathsf{B}
  \Rightarrow -\otimes^\bbL_A \mathsf{B}'$ (see \cite[\S2.3]{ICM-Keller}) and consequently to a morphism
  of bimodules
  ${}_{\nu_f}\mathsf{Bi}:{}_{-\otimes^\bbL_A \mathsf{B}}\mathsf{Bi}
  \Rightarrow {}_{-\otimes^\bbL_A \mathsf{B}'}\mathsf{Bi}$. 
\begin{lemma}\label{lem:triangulated-functor}
The above constructions give rise to a triangulated functor
\begin{equation}\label{eq:triang}
\rep(A,B) \too \rep(\perf_\dg(A),\perf_\dg(B))\,.
\end{equation}
\end{lemma}  
\begin{proof}
Note that when $f$ is a quasi-isomorphism, $\dgHo(\nu_f)$ is a natural isomorphism between triangulated functors. Using \cite[Lem.~9.8]{MT2}, one then concludes that ${}_{\nu_f}\mathsf{Bi}$ is a quasi-isomorphism. This implies that \eqref{eq:triang} is well-defined. The fact that it is triangulated is clear. 
\end{proof}
 \begin{proposition}\label{prop:key2}
The assignment $A \mapsto U(\perf_\dg(A))$ on objects and ${}_A \mathsf{B}_B \mapsto U({}_{(-\otimes^\bbL_A \mathsf{B})}\mathsf{Bi})$ on morphisms gives rise to a well-defined functor
\begin{equation}\label{eq:functor-real}
  \Az_0(X) \too \Hmo_0(k)\,.
\end{equation}
\end{proposition}
\begin{proof}
We start by verifying that the assignment $A \mapsto \perf_\dg(A)$ on objects and ${}_A\mathsf{B}_B \mapsto {}_{(-\otimes^\bbL_A \mathsf{B})}\mathsf{Bi}$ on morphisms gives rise to a well-defined functor from $\Az_0(X)$ to $\Hmo(k)$. Thanks to \eqref{eq:triang}, one has well-defined morphisms 
\begin{equation*}
\Hom_{\Az(X)}(A,B) \too \Hom_{\Hmo(k)}(\perf_\dg(A),\perf_\dg(B))\,.
\end{equation*}
Given bimodules ${}_A\mathsf{B}_B \in \rep(A,B)$ and ${}_B \mathsf{B}'_C \in \rep(B,C)$, the associativity of the (derived) tensor product gives rise to a canonical isomorphism of dg functors (and consequently to an isomorphism of bimodules)
\begin{eqnarray*}
(-\otimes^\bbL_A \mathsf{B}) \otimes^\bbL_B \mathsf{B}' \simeq -\otimes^\bbL_A (\mathsf{B} \otimes^\bbL_B \mathsf{B}') && {}_{(-\otimes^\bbL_A \mathsf{B}) \otimes^\bbL_B \mathsf{B}'}\mathsf{Bi} \simeq {}_{-\otimes^\bbL_A (\mathsf{B} \otimes^\bbL_B \mathsf{B}')}\mathsf{Bi}\,.
\end{eqnarray*}
This shows that the assignment ${}_A\mathsf{B}_B \mapsto {}_{(-\otimes^\bbL_A \mathsf{B})}\mathsf{Bi}$ preserves the composition operation. The
identities are also preserved since the $A\text{-}A$-bimodule
${}_AA_A$ is mapped to the identity bimodule ${}_{\id}\mathsf{Bi} =
\id_{\perf_\dg(A)}$. In conclusion, we obtain a functor 
\begin{equation}\label{eq:functor-new}
\Az(X) \too \Hmo(k)\,.
\end{equation}
Now, from Lemma~\ref{lem:triangulated-functor} and from the construction of the categories $\Az_0(X)$ and $\Hmo_0(k)$, one concludes that the searched functor
\eqref{eq:functor-real} is the additivization of \eqref{eq:functor-new}. This completes the proof.
\end{proof}

We now have all the ingredients needed for the conclusion of the proof of Theorem~\ref{thm:main}. By $\bbZ[1/r]$-linearizing the above functor \eqref{eq:functor-real} one obtains the following commutative diagram
$$
\xymatrix{
\Az_0(X) \ar[d]_-{(-)_{1/r}} \ar[r]^-{\eqref{eq:functor-real}}& \Hmo_0(k) \ar[d]^-{(-)_{\bbZ[1/r]}} \\
\Az_0(X)_{1/r} \ar[r]_-{\eqref{eq:functor-real}} & \Hmo_0(k)_{\bbZ[1/r]}\,.  
}
$$
Hence, the image of the isomorphism $[{}_{\cO_X}A_A]_{1/r}:\cO_X \stackrel{\sim}{\to} A$ of Proposition~\ref{prop:key} under \eqref{eq:functor-real} identifies with the following isomorphism
$$ U(-\otimes_{\cO_X}A)_{\bbZ[1/r]}:U(\perf_\dg(X))_{\bbZ[1/r]} \stackrel{\sim}{\too} U(\perf_\dg(A))_{\bbZ[1/r]}\,.$$
This completes the proof.
\section{Proof of Theorem~\ref{thm:nilpotence}}
Let us write $\pi:S \twoheadrightarrow S/I$ for the quotient map. One needs to show that it yields an isomorphism $U(\underline{\pi})_R:U(\underline{S})_R \stackrel{\sim}{\to} U(\underline{S/I})_R$. By the Yoneda lemma for the full subcategory of $\Hmo_0(k)_R$ containing the objects $U(\underline{S})_R$ and $U(\underline{S/I})_R$, one observes that it suffices to show that the induced homomorphism
$$(U(\underline{\pi})_R)_\ast:\Hom_{\Hmo_0(k)_R}(U(\underline{T})_R, U(\underline{S})_R) \too \Hom_{\Hmo_0(k)_R}(U(\underline{T})_R, U(\underline{S/I})_R)$$
is an isomorphism for $T=S,S/I$. Concretely, is suffices to show that 
\begin{equation}\label{eq:iso-2}
[-\otimes^\bbL_S {}_\pi \mathsf{Bi}]: K_0(\rep(\underline{T},\underline{S}))_R \too K_0(\rep(\underline{T},\underline{S/I}))_R
\end{equation}
is an isomorphism. The proof is now divided into two cases.
\subsection*{Case 1}{($S/J(S)$ $k$-separable)}
Assume that $S/I$ has finite global dimension and that $S/J(S)$ is $k$-separable. Since $(S/I)/J(S/I)=S/J(S)$, one concludes then form \cite[page~2]{Buchweitz} that the dg categories $\underline{S}$ and $\underline{S/I}$ are smooth. They are also proper, and so thanks to description \eqref{eq:smooth-case} the induced homomorphism \eqref{eq:iso-2} reduces to 
\begin{equation}\label{eq:iso-3}
[-\otimes^\bbL_S {}_\pi \mathsf{Bi}]: K_0(T^\op \otimes S)_R \too K_0(T^\op \otimes S/I )_R\,.
\end{equation}
Now, recall that by assumption $I$ is nilpotent. As a consequence, the (two-sided) ideal of the quotient map $T^\op \otimes S \twoheadrightarrow T^\op \otimes (S/I)$ is also nilpotent. Using the invariance of the Grothendieck group functor with respect to nilpotent extensions (see \cite[page~70]{Weibel}), we hence conclude that \eqref{eq:iso-3} is an isomorphism.
\subsection*{Case 2}{($1/p \in R$)}
Assume that $S/I$ has finite global dimension and that $k$ is a field of characteristic $p>0$ such that $1/p \in R$. Note that since $S$ and $S/I$ are finite dimensional and of finite global dimension we have the equivalences
\begin{eqnarray*}
\rep(\underline{T},\underline{S})\simeq \cD^b(\mathrm{mod}(T^\op \otimes S)) && \rep(\underline{T},\underline{S/I}) \simeq \cD^b(\mathrm{mod}(T^\op \otimes S/I))\,,
\end{eqnarray*}
where $\cD^b(\mathrm{mod}(-))$ stands for the bounded derived category of finitely generated modules. The above homomorphism \eqref{eq:iso-2} identifies then with 
\begin{equation*}
[-\otimes^\bbL_S {}_\pi \mathsf{Bi}]: G_0(T^\op \otimes S)_R \too G_0(T^\op \otimes S/I)_R\,.
\end{equation*}
Now, consider the following quotient maps
\begin{eqnarray*}
r_S: S \twoheadrightarrow S/J(S) & q_T:T \twoheadrightarrow T/J(T) & q_{S/I}:S/I \twoheadrightarrow S/J(S/I)=S/J(S)\,.
\end{eqnarray*}
Using the following commutative diagram\footnote{Note that although $T^\op\otimes S$ may have infinite global dimension, it is still true that $(T/J(T))^\op\otimes (S/J(S))$ has finite projective dimension over $T^\op\otimes S$.}
$$
\xymatrix{
G_0(T^\op \otimes S)_R \ar[rr]^-{[- \otimes_S^\bbL {}_\pi \mathsf{Bi}]} \ar@/_2.3pc/[ddrr]_-{[-\otimes^\bbL_{T^\op \otimes S} ({}_{q_T}\mathsf{Bi} \otimes {}_{q_S}\mathsf{Bi})]\quad} && G_0(T^\op \otimes (S/I))_R \ar[dd]^-{[-\otimes^\bbL_{T^\op \otimes (S/I)} ({}_{q_T}\mathsf{Bi} \otimes {}_{q_{S/I}}\mathsf{Bi})]} \\
&&\\
&& G_0((T/J(T))^\op \otimes (S/J(S)))_R
}
$$
one observes that it suffices to prove that 
\begin{equation}\label{eq:iso-4}
[-\otimes^\bbL_{T^\op \otimes S} ({}_{q_T}\mathsf{Bi} \otimes {}_{q_S}\mathsf{Bi})]:G_0(T^\op \otimes S)_R \too G_0((T/J(T))^\op \otimes (S/J(S)))_R
\end{equation}
is an isomorphism. Moreover, it is sufficient by base-change to treat the case $R=\bbZ[1/p]$. Note that the kernel of the quotient map $q_T^\op \otimes q_S$ is nilpotent. Hence, $G_0(T^\op \otimes S)$ and $G_0((T/J(T))^\op \otimes (S/J(S))$ are free $\bbZ$-modules with a basis given by the simple $((T/J(T))^\op \otimes (S/J(S))$-modules. In particular, they have the same rank. As a consequence, it suffices to prove that \eqref{eq:iso-4} (with $R=\bbZ[1/p]$) is a surjection. In order to do so, we consider the following commutative diagram
$$
\xymatrix{
K_0(T^\op \otimes S)_{\bbZ[1/p]} \ar[d] \ar[rrr]_-\sim^-{-\otimes^\bbL_{T^\op \otimes S} ({}_{q_T}\mathsf{Bi} \otimes {}_{q_S}\mathsf{Bi})} &&& K_0((T/J(T))^\op \otimes (S/J(S)))_{\bbZ[1/p]} \ar[d]^-\sim \\
G_0(T^\op \otimes S)_{\bbZ[1/p]} \ar[rrr]^-\sim_-{\eqref{eq:iso-4}} &&&G_0((T/J(T))^\op \otimes (S/J(S)))_{\bbZ[1/p]}\,.
}
$$
As in the proof of Case 1, the upper horizontal map is an isomorphism. Thanks to Proposition \ref{prop:vdB} below (with $U=T/J(T)$ and $U'=S/J(S)$) the right vertical map is also an isomorphism. Using these isomorphisms and  the commutativity of the above diagram we conclude that \eqref{eq:iso-4} is a surjection. This finishes the proof.

\begin{proposition}\label{prop:vdB}
Given a field $k$ of characteristic $p>0$, the induced map 
\begin{equation}\label{eq:induced-final}
K_0(U\otimes U')_{\bbZ[1/p]} \too G_0(U\otimes U')_{\bbZ[1/p]}
\end{equation}
is an isomorphism for any two finite dimensional semi-simple $k$-algebras $U$ and $U'$.
\end{proposition}
\begin{proof}
Since $U\otimes U'$ is finite dimensional, $K_0(U\otimes U')_{\bbZ[1/p]}$ and $G_0(U\otimes U')_{\bbZ[1/p]}$ are free $\bbZ$-modules with the same rank. Hence, it suffices to show that \eqref{eq:induced-final} is surjective. One can (and will) assume without loss of generality that $U$ and $U'$ are indecomposable. Let $Z$ (resp. $Z'$) be the center of $U$ (resp. of $U'$) and $Z_0$ (resp. $Z'_0$) the separable closure of $k$ in $Z$ (resp. in $Z'$). Under these notations, one has $Z_0\otimes Z'_0 = \bigoplus_i W_i$ with $W_i/k$ a separable field extension. As a consequence, one obtains the following equalities:
\begin{eqnarray*}
U\otimes U' & = & U \otimes_{Z_0}(Z_0 \otimes Z'_0) \otimes_{Z'_0} U' \\
& = & \bigoplus_i (U \otimes_{Z_0} W_i \otimes_{Z'_0} U') \\
&=& \bigoplus_i (U \otimes_{Z'_0} W_i) \otimes_{W_i} (U'\otimes_{W_0}W_i)\,.
\end{eqnarray*}
Replacing $k$ by $W_i$ and $U$ (resp. $U'$) by $U\otimes_{Z_0}W_i$ (resp. by $U'\otimes_{Z'_0} W_i$) one can (and will) assume that $Z$ and $Z'$ are purely inseparable $k$-algebras. We hence have
$$ U \otimes U' =(U \otimes Z') \otimes_{(Z\otimes Z')} (U' \otimes Z')\,.$$
Note that $D:=U \otimes U'$ is the tensor product of two Azumaya algebras over $W:=Z\otimes Z'$, and hence is itself and Azumaya algebra. Thanks to Lemma~\ref{lem:aux} below, $W$ is a local $k$-algebra. By lifting idempotents and invoking Morita equivalence the problem of showing that \eqref{eq:induced-final} is surjective can be reduced to the case that $D/J(D)$ is a division algebra. In this case, $D/J(D)= W/J(W) \otimes_W D$ is the unique simple $D$-module. Invoking Lemma~\ref{lem:aux} again, we find that $D=W\otimes_W D$ is an extension of $p^n$ copies (for some $n$) of $W/J(W) \otimes_W D$. Hence, $p^n[D/J(D)]$ is in the image of \eqref{eq:induced-final}. This finishes the proof.
\end{proof}
\begin{lemma}\label{lem:aux}
Let $k$ be a field of characteristic $p>0$ and $Z/k, Z'/k$ two purely inseparable field extensions. Under these assumptions, $Z\otimes Z'$ is a local $k$-algebra and its length (as a module over itself) is a power of $p$.
\end{lemma}
\begin{proof}
Note that if $e \in Z \otimes Z'$, then $e^{p^n} \in k$ for some $n \gg 0$. In the case where $e$ is an idempotent we then conclude that $e=0,1$. This implies that $Z\otimes Z'$ is a local $k$-algebra. It is also clear that the length of $Z\otimes Z'$ must divide $\mathrm{dim}_k(Z\otimes Z')$. Hence, it is necessarily a power of $p$.
\end{proof}

\end{document}